\documentclass[graybox]{svmult}

\usepackage{amsfonts}
\usepackage{color}
\usepackage{amsmath}
\usepackage{amstext}
\usepackage{graphicx}
\usepackage{amssymb}
\usepackage{esint}
\usepackage[all]{xy}
\usepackage{hyperref}

\usepackage{mathrsfs}

\usepackage{mathptmx}       % selects Times Roman as basic font
\usepackage{helvet}         % selects Helvetica as sans-serif font
\usepackage{courier}        % selects Courier as typewriter font
\usepackage{type1cm}        % activate if the above 3 fonts are
\usepackage{makeidx}         % allows index generation
\usepackage{graphicx}        % standard LaTeX graphics tool
\usepackage{multicol}        % used for the two-column index
\usepackage[bottom]{footmisc}% places footnotes at page bottom

\makeindex             % used for the subject index
\begin{document}

\title*{Lectures on BCOV Holomorphic Anomaly Equations}
\author{Atsushi Kanazawa and Jie Zhou}
\institute{Atsushi Kanazawa \at Department of Mathematics, University of Tokyo,
3-8-1 Komaba, Meguro-ku, Tokyo, 153-8914, Japan, \email{kanazawa@10.alumni.u-tokyo.ac.jp}
\and Jie Zhou \at Department of Mathematics, Harvard University, One Oxford Street, MA 02138, USA, \email{jiezhou@math.harvard.edu}}

\maketitle

\abstract{
The present article surveys some mathematical aspects of the BCOV holomorphic anomaly equations introduced by Bershadsky, Cecotti, Ooguri and Vafa \cite{BCOV,BCOV2}.
It grew from a series of lectures the authors gave at the Fields Institute in the Thematic Program of Calabi--Yau Varieties in the fall of 2013.}

\section{Introduction} \label{Introduction}

The present article is a gentle introduction to some mathematical aspects of the BCOV holomorphic anomaly equations \cite{BCOV,BCOV2},
which represent a beautiful generalization of the classical $g=0$ mirror symmetry \cite{CdOGP}.
The classical $g=0$ mirror symmetry states that counting the rational curves in a Calabi--Yau threefold $X^\vee$ (A-model)
is equivalent to studying the variation of Hodge structures of its mirror Calabi--Yau threefold $X$ (B-model).
Higher genus mirror symmetry is concerned with counting the higher genus curves in a Calabi--Yau threefold.
While Gromov--Witten theory rigorously defines a mathematical theory of counting curves of any genus and thus higher genus A-model makes sense at all genera,
the higher genus B-model, a generalization of the theory of variation of Hodge structures, has been much more mysterious. \\

A candidate of the higher genus B-model was provided by Bershadsky, Cecotti, Ooguri and Vafa in the seminal papers \cite{BCOV,BCOV2} (BCOV theory).
Among other things, they derived a set of equations, now called the BCOV holomorphic anomaly equations.
The importance of these equations lies in the fact that they describe the anti-holomorphicity of the topological string amplitudes and, moreover,
recursively relate the genus $g$ topological string amplitude $\mathcal{F}_g$ to those of lower genera.
The new feature of higher genus mirror symmetry is that the theory is no longer governed by holomorphic objects
but by a mixture of holomorphic and anti-holomorphic objects in the controlled manner.
In fact, although the classical mirror symmetry can be understood in the context of variation of Hodge structures,
it seem that the BCOV theory cannot easily be captured by present mathematics. \\

Our primary goal is to give a soft introduction to the BCOV holomorphic anomaly equations and related topics,
about which many references are currently scattered throughout  journals.
We try to make our exposition as simple and motivating as possible, keeping in mind that they should be understandable by non-experts.
The choice of topics covered in this article is very limited and also influenced by the authors' taste.
The subject is very vivid and likely to get into new developments in the next few years,
and we hope that this article serves as an entry point for non-experts to learn the subject.\\

The layout of this article is as follows.
Section \ref{Special Kahler Geometry} is a brief summary of special K\"ahler geometry of the moduli space of complex structures of a Calabi--Yau threefold.
Special K\"ahler geometry is the basic language to formulate mirror symmetry.
Section \ref{Mirror Symmetry} is an overview of mirror symmetry from points of view of both physics and mathematics.
The key feature of higher genus $(g\ge 1)$ mirror symmetry is the presence of holomorphic anomaly.
In the BCOV theory, the holomorphic anomaly is controlled by the BCOV holomorphic anomaly equations.
Sections \ref{BCOV HAE} and \ref{Holomorphic Limit} explain the BCOV holomorphic anomaly equations and holomorphic limit respectively,
with a particular emphasis on the similarity with the theory of elliptic curves.
We close this article by providing some examples in Section \ref{Examples}.

\subsection*{Acknowledgement}
This article grew from a series of lectures the authors gave at the Fields Institute in the Thematic Program of Calabi--Yau Varieties in the fall of 2013.
It is a great pleasure to record our thanks to all the people who attended the lectures.
Among many others, we would like to express our particular gratitude to N. Yui for her hospitality during the program and her suggestion to write up this set of notes.
The authors thank M. Alim, H. Fuji, S. Hosono, M. Miura, E. Scheidegger and S.-T. Yau for very helpful discussions on the subject.
We are also grateful to A. Zinger for correspondence about the recent progress on $g=1$ mirror symmetry.

%%%%%%%%%%%%%%%%%%%%%%%%%%%%%%%%%%%%%%%%%%%%%%%%%%%%%%%%%%%%%%%%%%%%%%%%%%%%%%%%%%%%%%

%%%%%%%%%%%%%%%%%%%%%%%%%%%%%%%%%%%%%%%%%%%%%%%%%%%%%%%%%%%%%%%%%%%%%%%%%%%%%%%%%%%%%%
%%%%%%%%%%%%%%%%%%%%%%%%%%%%%%%%%%%%%%%%%%%%%%%%%%%%%%%%%%%%%%%%%%%%%%%%%%%%%%%%%%%%%%

\section{Special K\"ahler Geometry} \label{Special Kahler Geometry}
In this section, we give a brief summary of the basics of special K\"ahler geometry that we need throughout this article.
Special K\"ahler geometry is a basic computational tool used in the calculations in mirror symmetry.
This section also serves to set conventions and notations.
Standard references are \cite{Str, BCOV, Fre}.

\subsection{Special Coordinates and Prepotential} \label{special coordinates}
Let $\mathcal{M}$ be the moduli space of complex structures of a smooth Calabi--Yau threefold $X$ of dimension $n:=\dim \mathcal{M}=h^{2,1}(X)$.
The vector bundle $\mathcal{H}:=R^3 \pi_* \underline{\mathbb{C}}\otimes \mathcal{O}_{\mathcal{M}}$ of rank $2n+2$ 
comes equipped with the Gauss--Manin connection $\nabla$ and  the natural Hodge filtration $F^\bullet$ of weight 3.
The Hodge filtration $F^\bullet$ yields the smooth decomposition
$$
\mathcal{H}=\mathcal{H}^{3,0}\oplus \mathcal{H}^{2,1}\oplus \mathcal{H}^{1,2} \oplus \mathcal{H}^{0,3},
$$
where $\mathcal{H}^{p,q}:=F^p\mathcal{H} \cap \overline{F^{q}\mathcal{H}}$.
The holomorphic line bundle $\mathcal{L}:=F^3 \mathcal{H}=\mathcal{H}^{3,0}$ is called the vacuum bundle.
We also fix a reference point $[X]\in \mathcal{M}$ and smoothly identify\footnote{
We take a universal covering of $\mathcal{M}$ if necessary but most of what follows works in a local setting.} the fibers of $\mathcal{H}$ with $H^3(X,\mathbb{C})$.
We endow $H^3(X,\mathbb{C})$ with the symplectic pairing $(\alpha, \beta):= \sqrt{-1} \int_{X} \alpha \cup \beta$.
Then the period domain ${\mathcal D}$ is defined by
\[
{\mathcal D}:=\left\{ [\omega] \in \mathbb{P}(H^3(X,\mathbb{C})) \ | \
(\omega, \omega )=0, (\omega, \overline \omega ) >0 \right\}.
\]
The period map ${\mathcal P}: \mathcal{M} \rightarrow {\mathcal D}$ assigns to $z=[X_z] \in \mathcal{M}$
the line $\mathcal{L}_z \subset H^3(X_z,\mathbb{C}) \cong H^3(X,\mathbb{C})$.
More concretely, by fixing a symplectic basis $\{ \alpha_I, \beta^J \}_{I,J =0}^n$ of $H^3(X,\mathbb{Z})$ and its dual basis $\{ A^I, B_J \}_{I,J =0}^n$ of $H_3(X,\mathbb{Z})$,
the period map ${\mathcal P}$ is written in terms of a section $\Omega=\{\Omega_z\}_{z\in \mathcal{M}}$ 
of the vacuum bundle $\mathcal{L}$ as\footnote{We use the Einstein summation convention.}
$$
\mathcal{P}(z) := \phi^I(z) \alpha_I + F_J(z) \beta^J,
$$
where $\phi^I(z) :=\int_{A^I} \Omega_z$ and $F_J(z) :=\int_{B_J} \Omega_z$.

\begin{proposition} \label{Special coord and prepotential}
With the notation above, the following hold:
\begin{enumerate}
\item The map $z \mapsto [\phi^0(z), \cdots, \phi^n(z)] \in \mathbb{P}^n$ is locally bi-holomorphic,
i.e. $\{\phi^I\}_{I=0}^n$ locally form homogeneous coordinates of the moduli space $\mathcal{M}$ around $z$.
\item Locally there exists a function $F(\phi)$ such that $F_J(z) = \frac{\partial F(\phi)}{\partial \phi^J}$ for $0 \le J \le n$.
\item $F(\phi)$ is holomorphic and  homogeneous of degree 2 in the variables $\phi$.
In particular $F(\phi) \in \Gamma(\mathcal{M},\mathcal{L}^2)$.
\end{enumerate}
\end{proposition}
\begin{proof}
We will show the second and third assertions and refer the reader to \cite{BG} for a proof the first assertion.
The following identity is useful in the computation below:
$$
\int_X \Omega_1 \cup \Omega_2=\int_{A^I}\Omega_1 \int_{B_I}\Omega_2-\int_{A^I}\Omega_2 \int_{B_I}\Omega_1
$$
for $\Omega_i \in H^3(X,\mathbb{C}) \ (i=1,2)$.
By the property of the Gauss--Manin connection, $\nabla_I:=\nabla_{\phi_I}$, we have
$$
\omega_I:=\nabla_I \Omega=\alpha_I+\frac{\partial F_J}{\partial \phi_I}\beta^J.
$$
Moreover, the Griffith transversality implies that $\{\omega^I\}_{I=0}^n$ form a basis of $\mathcal{H}^{3,0}_z\oplus \mathcal{H}^{2,1}_z$.
Then the relation $\int_X \omega_I \cup \omega_J=0$ yields
$$
\frac{\partial F_J}{\partial\phi^I}=\int_{B^I}\omega_J=\int_{B^J}\omega_I=\frac{\partial F_I}{\partial\phi^J},
$$
which shows there locally exists a function $F(\phi)$ such that $F_I=\frac{\partial F}{\partial \phi^I}$.
The function $F_I$ is linear because
$$
0=\int_X \Omega \cup \omega_I =\phi^J\nabla_JF_I - F_I.
$$
Therefore we conclude that $F(\phi)$ is homogeneous of degree 2 in $\phi$, a section of $\mathcal{L}^{2}$.
\qed
\end{proof}
The above local coordinates $\{\phi^i/\phi^0\}_{i=1}^n$ are often called special coordinates on the moduli space $\mathcal{M}$.
They are an example of canonical coordinates around a large complex structure limit (Section \ref{Kahler normal coordinates}) and play an important role in mirror symmetry.

%%%%%%%%%%%%%%%%%%%%%%%%%%%%%%%%%%%%%%%%%%%%%%%%%%%%%%%%%%%%%%%%%%%%%%%%%%%%%%%%%%%%%%%%%%%%%%%%%%%%%%%%%%%%%

\subsection{Special K\"ahler Manifolds}\label{Special Kahler Manifolds}

\begin{definition}
A Hodge manifold $M$ is a compact K\"ahler manifold with a Hermitian line bundle $(L,\langle *,**\rangle)$
such that a K\"ahler potential $K$ is given by $K=-\log \big\lVert\Omega\big\lVert$ where $\Omega$ is a local holomorphic section of $L$.
\end{definition}

Given a Hodge manifold $M$ with local coordinates $\{z_i\}_{i=1}^n$, the K\"ahler metric $G_{i\bar j}$, 
Christoffel symbols $\Gamma_{ij}^k$ and the curvature $R_{k i \bar j}^l$ are respectively given by
$$
G_{i\bar j}:=\partial_{i} \partial_{\bar j} K, \ \ \ \Gamma_{ij}^k:=G^{k\bar k}\partial_i G_{j\bar k},  \ \ \ R_{k i \bar j}^l:=-\partial_{\bar j}\Gamma_{ik}^l,
$$
where $(G^{i\bar j})$ is the inverse of the metric $(G_{i\bar j})$.

\begin{definition}
A special K\"ahler manifold $M$ is a Hodge manifold satisfying the following conditions:
\begin{enumerate}
\item Let $H$ be the vector bundle defined by
$$
H:=L\oplus (L\otimes TM)\oplus \overline{L\otimes TM} \oplus \bar L.
$$
There exists a connection $\mathbb{D}:\Gamma(M,H)\rightarrow \Gamma(M,H)\otimes \varOmega_{M}$ of the form
\begin{align}
\mathbb{D}_i \xi_0            & =\partial_i \xi_0               +\partial_i K \xi_0  +                                         C_{i0}^A \xi_A \notag \\
\mathbb{D}_i \xi_j               & =\partial_i \xi_j                +\partial_i K \xi_j   - \Gamma_{ij}^k\xi_k  + C_{ij}^A \xi_A \notag \\
\mathbb{D}_i \xi_{\bar j} & =\partial_i \xi_{\bar j}    +                                                                           C_{i\bar j}^A \xi_A \notag \\
\mathbb{D}_i \xi_{\bar 0} & =\partial_i \xi_{\bar 0} +                                                                          C_{i\bar 0}^A \xi_A, \notag
\end{align}
for a section $(\xi_A):=(\xi_0,\xi_j,\xi_{\bar j},\xi_{\bar0}) \in \Gamma(M,H)$.
We have a similar equations for $\mathbb{D}_{\bar i} \xi_A$.
\item $\mathbb{D}$ is flat: $[\mathbb{D}_i,\mathbb{D}_j]=[\mathbb{D}_{\bar i},\mathbb{D}_{\bar j}]=[\mathbb{D}_i,\mathbb{D}_{\bar j}]=0$.
\end{enumerate}
\end{definition}
Let $C_i:=(C_{i,A}^B)$ and write $\mathbb{D}_i=D_i+C_i$.
The condition $[\mathbb{D}_i,\mathbb{D}_j]=0$ implies that there exists a section $F \in \Gamma(M,L^{2})$ such that
$$
C_{ijk}:=-D_iD_jD_k F\,,
$$
and that
$$
C_i=\begin{bmatrix}
        0  &  -\delta_i^j  & 0  & 0 \\
        0  &  0  &  e^KG^{k\bar k}C_{ijk}  &   0\\
        0  &  0   &  0  &  -G_{i\bar j} \\
        0  &  0  &  0  &  0
        \end{bmatrix}.
$$
The quantity $C_{ijk}$ is often called the B-model Yukawa coupling (Section \ref{Mirror Symmetry}).
We obtain a similar form for $C_{\bar i}$ from the condition $[\mathbb{D}_{\bar i},\mathbb{D}_{\bar j}]=0$.
The last condition $[\mathbb{D}_i,\mathbb{D}_{\bar j}]=0$ leads to the following:
\begin{equation}R_{i\bar j k \bar l}=G_{i \bar j}G_{k \bar l}+ G_{i \bar l}G_{k \bar j}-e^{2K}C_{ikm}C_{\bar j \bar l \bar n} G^{m \bar n}. \label{Special Kahler geometry relation}
\end{equation}
This relation is called the special K\"ahler geometry relation. \\

The most important example for us of a special K\"ahler manifold is the moduli space $\mathcal{M}$ of complex structures of a smooth Calabi--Yau threefold $X$.
We define a Hermitian metric, called the the $tt^*$-metric, $\langle*,** \rangle $ on $\mathcal{H}$ by
$$
\langle \xi, \eta \rangle:=-\int_X C(\xi) \cup \bar \eta,
$$
where $C$ is the Weil operator.
The Hodge manifold structure on $\mathcal{M}$ is given by the K\"ahler potential
$$
K(z,\bar z):= - \log \big\lVert\Omega\big\lVert=-\log \sqrt{-1}\int_X \Omega \cup \overline{\Omega},
$$
where $\Omega$ is a local holomorphic section of the vacuum bundle $\mathcal{L}$.
The induced K\"ahler metric is called  the Weil--Petersson metric in this case.
Moreover, we have a canonical isomorphism
$$
\mathcal{H} \cong \mathcal{L}\oplus (\mathcal{L}\otimes T\mathcal{M})\oplus \overline{\mathcal{L}\otimes T\mathcal{M}} \oplus \overline{\mathcal{L}}.
$$
because the fiber of the RHS over $[X] \in \mathcal{M}$ is naturally identified with
$$
H^{3,0}(X)\oplus H^{2,1}(X)\oplus H^{1,2}(X)\oplus H^{0,3}(X),
$$
where we used the Kodaira--Spencer map for $H^{2,1}(X) \cong (\mathcal{L}\otimes T\mathcal{M})|_{[X]}$.
The vector bundle $\mathcal{H}$ admits the Gauss--Manin connection, which is flat and satisfies the Griffith transversality condition. \\

It is instructive to show how the above data endows $M$ with a special K\"ahler structure.
The Kodaira--Spencer map gives rise to a homomorphism
$$
C:T\mathcal{M} \longrightarrow \oplus_{p=3}^0 \mathrm{Hom}(\mathcal{H}^{p,3-p},\mathcal{H}^{p-1,4-p}).
$$
We define $C_i:=C(\frac{\partial}{\partial z_i})$ for local coordinates $\{z^i\}_{i=1}^n$ of $\mathcal{M}$.
We also define $D_i:=D(\frac{\partial}{\partial z_i})$, where $D$ the $(1,0)$-component of the covariant derivative with respect to the $tt^*$-metric ($tt^*$-connection).
The notations $\overline{C}_{\bar i}$ and $\overline{D}_{\bar i}$ are defined in a similar manner.
It is a good exercise to check that
$$
\nabla^{1,0}=D+C, \ \ \ \nabla^{0,1}=\overline{D}+\overline{C},
$$
where $\nabla^{1,0}$ is the $(1,0)$-component of the Gauss--Manin connection $\nabla$ and similar for $ \nabla^{0,1}$.

\begin{proposition}[Cecotti--Vafa \cite{CV2}]
The $tt^*$-connection and the matrix $C$ satisfy the following set of equations, called the $tt^*$-equations.
\begin{align}
[D_i,D_j]=[\overline{D}_{\bar i},\overline{D}_{\bar j}]=0, \ \ & \ [D_i,\overline{C}_{\bar j}]=[\overline{D}_{\bar i},C_j]=0 .\notag \\
[D_i,C_j]=[D_j,C_i], \ \ \ [\overline{D}_{\bar i},\overline{C}_{\bar j}]=&[\overline{D}_{\bar j},\overline{C}_{\bar i}], \ \ \
[D_i,\overline{D}_{\bar j}]=-[C_i,\overline{C}_{\bar j}]. \notag
\end{align}
\end{proposition}
\begin{proof}
The relations $[D_i,D_j]=[\overline{D}_{\bar i},\overline{D}_{\bar j}]=0$ follows from the fact that the curvature of the $tt^*$-connection is of type $(1,1)$.
The rest of the equations follows from a detailed study of the Gauss--Manin connection.
We refer the reader to \cite{CV2,BCOV, HKKPTVVZ} for a proof.
\qed
\end{proof}
The $tt^*$-equations are equivalent to the existence of a family of flat connections on $\mathcal{H}$ of the form:
$$
\nabla^\alpha=D+\alpha C, \ \ \ \overline{\nabla}^\alpha=D+\alpha^{-1}C
$$
for an arbitrary constant $\alpha \in \mathbb{C}$.
For $\alpha=1,$ we recover the Gauss--Manin connection.
In this situation, the section $F \in \Gamma(\mathcal{M},\mathcal{L}^2)$ is the one obtained in Proposition \ref{Special coord and prepotential}.\\

Let $e_0$ be a local section of the vacuum bundle $\mathcal{L}$, then $\{e_i:=C_ie_0\}_{i=1}^n$ form a local frame of $\mathcal{H}^{2,1}$.
Therefore, their complex conjugates $\{\bar e_{\bar i}\}_{\bar{i}=0}^n$ form a local frame of $\mathcal{H}^{1,2}\oplus \mathcal{H}^{0,3}$.
We denote by $g_{i \bar j}:=\langle e_i, \bar{e}_{\bar j}\rangle$ the $tt^*$-metric with respect to this frame. %, i.e.
%$$
%g_{0\bar 0}:=\langle e_0, \bar{e}_{\bar{0}}\rangle, \ \ \ g_{i\bar j}:=\langle e_i, \bar{e}_{\bar j}\rangle.
%$$
It is worth noting that the $tt^*$-connection is nothing but the induced connection from the connection on $\mathcal{L}$ by the Hermitian metric $e^{-K}$
and the connection on $T\mathcal{M}$ by the Weil--Petersson metric.
In fact, the Weil--Petersson metric is related to the $tt^*$ metric by
$$
G_{i \bar j}=\partial_i \overline{\partial}_{\bar j} (-\log g_{0\bar0})=\frac{g_{i \bar j}}{g_{0\bar0}}, \ \ \ (1\le i,j \le n). 
$$
Now the special K\"ahler geometry relation in (\ref{Special Kahler geometry relation}) follows from
a direct computation of the $tt^*$-equations in terms of the local frame $\{e_i,\bar e_{\bar{i}} \}_{i=0}^n$. 
First, since $\{e_{\bar{i}} \}_{i=1}^n$ form a local frame of $\mathcal{H}^{1,2}$, we can write $C_ie_j=C_{ij}^{\bar k}e_{\bar k}$. 
We also have
$$
[D_i,\overline{D}_{\bar j}]e_0=G_{i \bar j}e_0, \ \ \
[C_i,\overline{C}_{\bar j}]e_0=-\overline{C}_{\bar j}e_i,
$$
and thus $\overline{C}_{\bar j}e_i=G_{i \bar j}e_0$.
Next, we have
\begin{align}
[D_i,\overline{D}_{\bar j}]e_k&=-\partial_{\bar j}(g^{\bar m l}\partial_i g_{k \bar m})=(R^l_{ki\bar j} + G_{i\bar j}\delta^l_k)e_l, \notag \\
[C_i,\overline{C}_{\bar j}]e_k&= C_i(G_{k \bar j }e_0) -C_{\bar j} C_{ik}^{\bar m} \bar e_{\bar m}=G_{k\bar j}e_i - C_{ik}^{\bar m}C_{j \bar m}^l e_l, \notag
\end{align}
and thus the special K\"ahler relation in (\ref{Special Kahler geometry relation}).
Here raising and lowering indices are given by the metric $(g_{i \bar j})_{i,j =1}^n$.
For example, we define a quantity
$$
C^{ij}_{\bar k}:=C_{\bar i \bar j \bar k}g^{i\bar i}g^{j\bar j}=e^{2K} C_{\bar i \bar j \bar k} G^{i\bar i} G^{j \bar j},
$$
which will appear in the BCOV holomorphic anomaly equations.
%%%%%%%%%%%%%%%%%%%%%%%%%%%%%%%%%%%%%%%%%%%%%%%%%%%%%%%%%%%%%%%%%%%%%%%%%%%%%%%%%%%%%%%%%%%%%%%%%%%%%%%%%%%%%
%%%%%%%%%%%%%%%%%%%%%%%%%%%%%%%%%%%%%%%%%%%%%%%%%%%%%%%%%%%%%%%%%%%%%%%%%%%%%%%%%%%%%%%%%%%%%%%%%%%%%%%%%%%%%

\section{Mirror Symmetry} \label{Mirror Symmetry}
Since its discovery, mirror symmetry has played one of the central roles in the interface between superstring theory and mathematics.
It originates from representations of the $N=2$ superconformal algebra and studies the interplay between two different combinations of chiral states in the left- and right-moving sectors.
Mirror symmetry in mathematics comes from a realization of the the $N=2$ superconformal fields theory as a non-linear $\sigma$-model on a Calabi--Yau threefold.
The process of building a mathematical foundation of mirror symmetry has given impetus to new fields in mathematics,
such as Gromov--Witten theory, quantum cohomology and Fukaya category \cite{CK,HKKPTVVZ}.

%%%%%%%%%%%%%%%%%%%%%%%%%%%%%%%%%%%%%%%%%%%%%%%%%%%%%%%%%%%%%%%%%%%%%%%%%%%%%%%%%%%%%%%%%%%%%%%%%%%%%%%%%%%%%
\subsection{Gromov--Witten Potentials}
Gromov--Witten theory lays a mathematical foundation of a curve counting theory.
For a Calabi--Yau threefold $X^\vee$, we define the genus $g$ Gromov--Witten invariant $N_{g}(\beta)$ of $X^\vee$ in the curve class $\beta \in H_{2}(X^\vee,\mathbb{Z})$ by
$$
N_{g}(\beta):=\int_{[\overline{M}_{g}(X^\vee,\beta)]^{vir}}1,
$$
Here $[\overline{M}_{g}(X^\vee,\beta)]^{vir}$ is the virtual fundamental class of the coarse moduli space of stable maps $\overline{M}_{g}(X^\vee,\beta)$
of the expected dimension, which is $0$ for a Calabi--Yau threefold.
Let $\{T_1,\dots,T_{h^{1,1}}\}$ be a basis of $H^2(X^\vee,\mathbb{Z})$.
Then the genus $g$ Gromov--Witten potential ${\tt F}_g(t)$ of $X^\vee$ is defined by
\begin{eqnarray}\label{AmodelFg}
{\tt F}_0(t)&:=&\frac{1}{6}\int_{X^\vee}(t^iT_i)^3+\frac{1}{24}  \int_{X^\vee}c_2(X^\vee) \cup t^iT_i-\frac{\chi(X^\vee)}{(2 \pi i)^3}\zeta(3)+\sum_{\beta \ne 0}N_0(\beta)q^\beta ,
\notag \\
{\tt F}_1(t)&:=&-\frac{1}{24}  \int_{X^\vee}c_2(X^\vee) \cup t^iT_i+ \sum_{\beta \ne 0}N_1(\beta)q^\beta ,
\notag \\
{\tt F}_g(t)&:=&\frac{\chi(X^\vee)}{2}(-1)^g\frac{|B_{2g}B_{2g-2}|}{2g(2g-2)(2g-2!)}
+\sum_{\beta \ne 0}N_g(\beta)q^\beta \ \ \ (g\ge 2).
\end{eqnarray}
where $B_k$ is the $k$-th Bernoulli number and $q:=e^{2\pi \sqrt{-1}t^i T_i}$ with the K\"ahler parameters $\{t^i\}_{i=1}^{h^{1,1}}$.
The constant term above represents the Gromov--Witten invariant $N_g(0)$ of degree 0,
the contribution from the constant maps\footnote{
We have $N_g(0)=\int_{\overline{M}_{g}\times X^\vee}c_{top}(Ob)=(-1)^g\frac{\chi(X^\vee)}{2}\int_{\overline{M}_{g}}c_{g-1}^3(\mathcal{H}_g)$,
where $Ob\rightarrow \overline{M}_{g,0}(X,0)\cong \overline{M}_{g}\times X^\vee$ is the obstruction bundle and $\mathcal{H}_g \rightarrow \overline{M}_{g}$ is the Hodge bundle \cite{FP}.}.
An important observation from superstring theory is that we should not consider each invariant $N_{g}(\beta)$ individually,
but consider them all together as a generating series.

%\begin{example}[Quintic threefold]
%Let $X^\vee \subset \mathbb{P}^4$ be the quintic threefold and $H$ the hyperplane class. 
%Then we have 
%$$
%{\tt F}_0(t)=\frac{5}{6}t^3+\frac{25}{12}+\frac{200}{(2 \pi i)^3}\zeta(3)+\sum_{d> 0}N_0(dH)q^d
%$$ 
%where $N_0(H)=2875$, $N_0(2H)=609250+\frac{2875}{2^3}$, $N_0(3H)=317206375+\frac{2875}{3^3}$ etc. 
%It is believed that $N_0(dH)$ is of the form 
%$$
%N_0(dH)=n_d+\sum_{k|d}\frac{n_{d/k}}{k^3}
%$$
%for {\it the number} $n_d\in \mathbb{N}$ of degree $d$ rational curves in $X^\vee$. 
%\end{example}

%%%%%%%%%%%%%%%%%%%%%%%%%%%%%%%%%%%%%%%%%%%%%%%%%%%%%%%%%%%%%%%%%%%%%%%%%%%%%%%%%%%%%%%%%%%%%%%%%%%%%%%%%%%%%
\subsection{Mirror Symmetry in Physics}
In this section we will give an overview of the physical origin of mirror symmetry.
This section is independent of other sections and can be skipped depending on the reader's background.
The exposition is based on \cite{Wit,BCOV, CK, HKKPTVVZ, Ali}. \\

We begin with a review of the $N=2$ superconformal field theory (SCFT).
One feature of the conformal field theory is that a field $\Phi(z,\bar z)$ factorizes into
the left- and right-moving part : $\Phi(z,\bar z)=\phi(z) \bar \phi(\bar z)$.
Therefore we obtain two copies of the $N=2$ conformal algebra and this is often referred to as the $N=(2,2)$ superconformal algebra.
More precisely, the $N=2$ SCFT consists of two conjugate left and right supersymmetries $G^\pm$ and $\overline{G}^\pm$,
and two $U(1)$ currents $J$ and $\overline{J}$.
Among the important commutation relations, we have
$$
(G^\pm)^2=0, \ \ \ \{G^+,G^-\}=2H_L, \ \ \ [G^{\pm},H_L]=0,
$$
where $H_L$ is the left-moving Hamiltonian, and parallel relations for the right movers.
A prototypical example of $N=2$ SCFT is the supersymmetric non-linear $\sigma$-model into a Calabi--Yau threefold.
To get a chiral ring, we need to consider suitable combinations of left- and right-moving supersymmetries.
There are two inequivalent choices, up to conjugation,
$$
Q_A:=G^++\overline{G}^+, \ \ \ Q_B:=G^++\overline{G}^-.
$$
The ring of the cohomology operators for $Q_A$ is called the $(c,c)$ ring and that for $Q_B$ is called the $(a,c)$ ring, where $a$ and $c$ stand for chiral and anti-chiral respectively.
As far as cohomology states are concerned $Q_A$ and $Q_B$ and their conjugates all give rise to an equivalent Hilbert space.
However, the rings of cohomology operators are different (via the state-operator correspondence).
The origin of mirror symmetry is the sign flip of the left moving current $J \leftrightarrow -J$, which is just a matter of convention.
Mirror symmetry relates the deformation of the $(a,c)$ chiral ring with that of the $(c,c)$ chiral ring as we will see below.  \\

Topological {\it string} theory is obtained by coupling the above theory with the {\it world-sheet gravity}.
This means that we integrate the correlation functions over the moduli spaces of Riemann surfaces.
In this case, the maps $\sigma:\Sigma_g\rightarrow Y$ from the world-sheet Riemann surfaces $\Sigma_g$ to a target space $Y$ are interpreted as Feynman diagrams in the string theory.
Here the target space $Y$ depends on the construction of $N=2$ SCFT.
In order to have globally defined charges on the Riemann surfaces, a {\it topological twist} is required  \cite{Wit}.
This makes $Q$ a scalar operator and also changes the $J$-charge of the chiral rings.
There are two types of topological twists called  the A-model and B-model corresponding to the choice of the scalar operator $Q=Q_A$ and $Q=Q_B$, respectively.
An advantage of the twisted topological theory lies in the fact that the physical states of the theory correspond to cohomology classes of $Q$
and that the path integral for a $Q$-invariant amplitude localizes to a sum of fixed points of the symmetry.
We can think of the twisted topological theory as extracting a certain class of supersymmetric ground states from the original SCFT.
In the $(c,c)$-twisting case (A-model), the topological correlation functions are sensible only to the K\"ahler class of $Y$ and compute the rational curves in $Y$.
On the other hand, in the $(a,c)$-twisting case (B-model), the topological correlation functions are sensible only to the complex structure of $Y$. \\

The space of ground-states $\mathcal{H}$ gives rise to a vector bundle over the moduli space $\mathcal{M}$ of the theory.
The vacuum state, which corresponds to the identity element in the chiral ring, varies over the moduli space and induces a splitting of the bundle $\mathcal{H}$,
which collects the states created by the chiral ring of $(J,\overline{J})$-charge $(1,1)$,
$\mathcal{H}=\oplus_{i=0}^3\mathcal{H}^{(i,i)}$
with the charge grading.
The chiral ring $\mathcal{H}$ has an associative multiplication $\circ$ described as follows.
We take a basis $\{\psi_{0},\psi_{a},\psi^{a},\psi^{0}\}_{a=1}^{\mathrm{dim}\mathcal{H}^{1,1}}$ of $\mathcal{H}$,
where $\psi_{0}$ is the identity operator of charge $(0,0)$ and $\psi_{a}$'s are of charge $(1,1)$,
and we require the basis to be symplectic with respect to topological metric (the topological correlation function on the sphere): $\langle \psi_a,\psi^b\rangle_0 =\delta_{a}^{b}, \langle\psi_0,\psi^0\ \rangle_0 =1$.
Then the ring structure, called a Frobenius structure, is given by
$$
\psi_{a}\circ\psi_{0}= \psi_{a},\ \ \
\psi_{a}\circ \psi_{b}=C_{abc}\psi^{c}, \ \ \
\psi_{a}\circ\psi^{b}=\delta_{a}^{b}\psi^{0}, \ \ \
\psi_{a}\circ\psi^{0}=0,
$$
where $C_{abc}:=\langle \psi_a,\psi_b,\psi_c\rangle_0$ are the 3-point functions on the sphere. \\

In the A-model realization ($Y$ K\"ahler or symplectic), the ring $\mathcal{H}$ is given by
\begin{equation*}
\mathcal{H}=H^{even}(Y,\mathbb{C})=\bigoplus_{d=1}^3H^{2d}(Y,\mathbb{C})\,,
\end{equation*}
with $\{\psi_{0},\psi_{a}\}_{a=1}^{h^{1,1}(Y)}$ the basis for $H^{0}(Y,\mathbb{C}), H^{1,1}(Y,\mathbb{C})$ respectively, and
$\{\psi^{a},\psi^{0}\}_{a=1}^{h^{1,1}(Y)}$ the dual basis. In this realization, the moduli space $\mathcal{M}$ is the moduli space of complexified K\"ahler structures of $Y$ (see \cite{Mor, CK} for example) and
$\{\psi_{a}\}$ provides a basis for the tangent space of $\mathcal{M}$.
The multiplication $\circ $ corresponds to the quantum product in the quantum cohomology ring $(H^{even}(Y,\mathbb{C}), \circ )$.
In fact, the structure constants $C_{abc}$ are the A-Yukawa couplings $K_{abc}$ in the $\{t^{a}\}_{a=1}^{h^{1,1}(Y)}$ coordinates
and are the generating function of genus zero Gromov--Witten invariants (with three insertions $\psi_{a}, \psi_{b},\psi_{c}$).\\

In the B-model realization ($Y$ Calabi--Yau), the ring $\mathcal{H}$ is given by
\begin{equation*}
\mathcal{H}= \bigoplus_{p=0}^{3}H^{p}(Y, \wedge^{p}TY)\cong \bigoplus_{p=0}^{3}H^{3-p,p}(Y,\mathbb{C})\,,
\end{equation*}
where the map between $H^{p}(Y, \wedge^{p}TY)$ and $H^{p}(Y,\Omega^{3-p})=H^{3-p,p}(Y,\mathbb{C})$
is obtained by taking the wedge product with a choice $\Omega$ for a section of the vacuum bundle.
Similar to the A-model, we can take the basis for $H^{0}(Y, \wedge^{0}TY),H^{1}(Y, \wedge^{1}TY)$ to be $\{\psi_{0},\psi_{a}\}_{a=1}^{h^{2,1}(Y)}$.
The moduli space $\mathcal{M}$ is the moduli space of complex structures of $Y$ and
$\{\psi_{a}\}_{a=1}^{h^{2,1}(Y)}$ provides a basis for the tangent space of $\mathcal{M}$, which is identified with $H^{1}(Y,TY)$ by the Kodaira--Spencer map.
Then the product $\circ$ becomes the wedge product in the cohomology.
In particular, the structure constants $C_{abc}$ are given by
\begin{equation*}
C_{abc}=-\int_{Y} \Omega \wedge \psi_{a}\psi_{b}\psi_{c}\Omega\,,
\end{equation*}
which are the normalized B-model Yukawa couplings in the special coordinates $\{t^{a}\}_{a=1}^{h^{2,1}(Y)}$. \\

A pair $(X^\vee,X)$ of Calabi--Yau threefolds is called a mirror pair if the A-model with target space $X^\vee$ is equivalent to the B-model with target space $X$, and vice versa.
The variation of the splitting is encoded in the Gromov--Witten invariants in the A-model.
In the B-model, we consider the non-holomorphic variation of Hodge structure (instead of holomorphic filtration) and we already see the origin of holomorphic anomalies here.
These two variations of the splittings are governed by the special K\"ahler geometry on the moduli spaces \cite{Str} . \\

The key observation \cite{BCOV,BCOV2} is the failure of decoupling of the two conjugate theories on $\Sigma_g$.
Due to this interaction, the topological string amplitude $\mathcal{F}_g$ should depend also on its conjugate coordinates in the following manner:
\begin{align}
\mathcal{F}_1:&=\int_{\overline{M_1}} \frac{d\tau d \bar \tau}{\Im(\tau)} \mathrm{Tr}(-1)^{F_L+F_R}F_LF_Rq^{L_0}\bar q^{\bar L_0},
\ \ \ q:=e^{2\pi \sqrt{-1} \tau},\notag \\
\mathcal{F}_g:&=\int_{\overline{M_g}}[dm^{i}d\bar m^{\bar i}]\langle \prod_{i=1}^{3g-3} (\int_{\Sigma_g} G^+ \mu_i) (\int_{\Sigma_g}G^- \bar \mu_{\bar i})\rangle \ \ (g\ge 2),\ \notag
\end{align}
where $F_{L},F_{R}$ are the fermion number operators, $\mu_i \in TM_g|_{\Sigma_g} \cong H^1(\Sigma_g,T\Sigma_g)$ is the Beltrami differential and $dm^{i}$ is the dual 1-form to $\mu_i$.
Then the anti-holomorphicity of $\mathcal{F}_g$ is measured by the boundary components of $\overline{M}_g$ corresponding to degenerate curves.
This leads us to the {\it BCOV holomorphic anomaly equations} (see Section \ref{BCOV HAE}):
$$
\partial_{\bar i}\mathcal{F}_{g}
=\frac{1}{2}C^{jk}_{\bar i} ( D_j D_k \mathcal{F}_{g-1} + \sum_{r=1}^{g-1} D_j \mathcal{F}_{r}\, D_k \mathcal{F}_{g-r}) \ \ \ (g\ge 2) .
$$
It is important that the equations is written in terms of special K\"ahler geometry,
in particular the Weil--Petersson geometry in the B-model, and thus things are easier to compute in the B-model.
Moreover, there is a procedure, called the {\it holomorphic limit} (Section \ref{Holomorphic Limit}), to obtain a holomorphic object.
For example, the Gromov--Witten potential is obtained as the holomorphic limit
$$
{\tt F}_g(t)=\lim_{\bar t \rightarrow \sqrt{-1}\infty}(\phi^{0})^{2g-2}\mathcal{F}_g(t,\bar t)
$$
of the topological string amplitude, where $\phi^{0}$ is the period integral described in Section \ref{special coordinates}.\\

We close this section by commenting on Witten's insight into the BCOV theory.
In \cite{Wit2}, he considered a Hilbert space obtained by geometric quantization of $H^3(X,\mathbb{R})$ as a symplectic phase space
and related it to the base-point independence of the total free energy $\mathcal{Z}=\sum_{g=0}^{\infty}\lambda^{2g-2}\mathcal{F}_{g}$ of the B-model on the family.
The background (base-point) independence of $\mathcal{Z}$ tells that it satisfies some wave-like equations on $\mathcal{M}$ arising from geometric quantization.
These equations are shown to be equivalent to the master anomaly equations \cite{BCOV2} for $\mathcal{Z}$,
which are identical to the set of holomorphic anomaly equations for the topological string amplitudes $\{\mathcal{F}_{g}\}_{g=0}^{\infty}$.

%%%%%%%%%%%%%%%%%%%%%%%%%%%%%%%%%%%%%%%%%%%%%%%%%%%%%%%%%%%%%%%%%%%%%%%%%%%%%%%%%%%%%%

\subsection{Mirror Symmetry in Mathematics}\label{MSinmath}
Mirror symmetry in a broad sense claims that, given a family of Calabi--Yau threefolds  $\mathcal{X}\to \mathcal{M}$ with a so-called large complex structure limit (LCSL, see for example \cite{Mor, CK} for details),
there exists another family $\mathcal{X}^\vee\to \mathcal{N}$ of Calabi--Yau threefolds such that complex geometry of $X$ is equivalent to symplectic geometry of $X^\vee$.
Here $X$ and $X^\vee$ are generic members of $\mathcal{X} \to \mathcal{M}$  and $\mathcal{X}^{\vee} \to \mathcal{N}$ respectively.
There are various version of mirror symmetry \cite{CK, HKKPTVVZ} and we will explain only one version of mirror symmetry below \cite{CdOGP}. \\

We begin with a formulation of $g=0$ mirror symmetry.
We will use the same notation as in Section \ref{Special Kahler Geometry}.
Let $[\phi^0,\dots,\phi^n]$ be the local projective coordinates around the LCSL of $\mathcal{M}$.
Assume that $A^0 \in H_3(X,\mathbb{Z})$ is the vanishing cycle at the LCSL of the family $\mathcal{X}\to \mathcal{M}$.
Then we define a local coordinates $\{t^i\}_{i=1}^n$ around the LCSL by
$$
[\phi^0(z),\dots,\phi^n(z)]=\phi^0(z)[1,t^1(z),\dots,t^n(z)],
$$
and introduce the mirror map by $q_i(z):=e^{2\pi \sqrt{-1}t^i}$.
The Picard--Fuchs system, together with the Griffith transversality condition, solve for the B-Yukawa couplings $C_{ijk}(z)$ of $X$.
The $g=0$ mirror symmetry claims that the A-Yukawa coupling of $X^\vee$
$$
K_{ijk}:=\frac{\partial}{\partial t^i}\frac{\partial}{\partial t^j}\frac{\partial}{\partial t^k} {\tt F}_0(t)
$$
is obtained by, together with the mirror map, the following:
\begin{equation}\label{normalizedYukawa}
K_{ijk}(q)=(\phi^0(z))^{-2}C_{lmn}(z)\frac{\partial z^l}{\partial t^i}\frac{\partial z^m}{\partial t^j}\frac{\partial z^n}{\partial t^k}.
\end{equation}
While this version of $g=0$ mirror symmetry conjecture is still open in general,
it is rigorously proven for a large class of Calabi--Yau threefolds independently by Givental \cite{Giv} and Lian--Liu--Yau \cite{LLY}. \\

We are now in a position to give a formulation of higher genus $(g\ge 1)$ mirror symmetry.
The classical $g=0$ mirror symmetry is concerned with counting rational curves in a given Calabi--Yau threefold $X^\vee$ and it is governed by Hodge theory of its mirror threefold $X$.
The main feature of higher genus mirror symmetry is that the theory is no longer governed by holomorphic objects
but a mixture of holomorphic and anti-holomorphic objects in a controlled manner.
It is safe to say that the mathematics involved in higher genus mirror symmetry has not well-understood at this point.
For example, we do not have a convenient mathematical definition\footnote{See \cite{Cos} which proposes a rigorous definition for the $\mathcal{F}_{g}$'s.}
of topological string amplitudes $\mathcal{F}_g$ for $g \ge 2$.
Despite some mathematical difficulty, higher genus mirror symmetry is summarized as follows:

\begin{conjecture}[Mirror Symmetry \cite{BCOV,BCOV2}]
Let $(X,X^\vee)$ be a mirror pair of Calabi--Yau threefolds.
Assume that a LCSL on the complex moduli space $\mathcal{M}$ of $X$ is chosen.
Then the following holds:
\begin{enumerate}
\item There exists a $C^\infty$-section $\mathcal{F}_g(z,\bar z)\in \Gamma_{C^\infty}( \mathcal{M},\mathcal{L}^{2-2g})$,
called the genus $g$ topological string amplitude.
\item There exist recursive equations, called BCOV holomorphic anomaly equations, which measure the anti-holomorphicity of $\mathcal{F}_g(z,\bar z)$:
\begin{align}
\partial_{i}\partial_{\bar j}\mathcal{F}_1 &=\frac{1}{2}C_{ijk}C^{kl}_{\bar j}+(1-\frac{\chi(X^\vee)}{24})G_{i \bar j}, \notag \\
\partial_{\bar i}\mathcal{F}_{g}
&=\frac{1}{2}C^{jk}_{\bar i} ( D_j D_k \mathcal{F}_{g-1} + \sum_{r=1}^{g-1} D_j \mathcal{F}_{r}\, D_k \mathcal{F}_{g-r}) \ \ \ (g\ge 2) .\notag
\end{align}
\item There exists a procedure, called the {\it holomorphic limit},  to obtain from $\mathcal{F}_g(z,\bar z)$
a holomorphic section $F_g(z)\in \Gamma( \mathcal{M},\mathcal{L}^{2-2g})$.
\item  The Gromov--Witten potential ${\tt F}_g(t)$ of $X^\vee$ is obtained by the following identity under the {\it mirror map}
$$
{\tt F}_g(t)=(\phi^0(z))^{2g-2}F_g(z),
$$
where the mirror map and the period $\phi^{0}(z)$ are taken at the LCSL.
\end{enumerate}
\end{conjecture}
The classical $g=0$ mirror symmetry also fits into this framework but without holomorphic anomaly, i.e. $\mathcal{F}_0(z,\bar z)=F(z)$.
The difficulty in higher genus mirror symmetry lies in the fact
that the BCOV holomorphic anomaly equations determine the topological string amplitude $\mathcal{F}_{g}(z,\bar z)$ only up to some holomorphic ambiguity $f_{g}(z)$.
For small genus $g$, the ambiguity can be fixed by the knowledge on the behavior of $\mathcal{F}_g$ at the various boundaries of the moduli space.
This is a rough sketch of higher genus mirror symmetry.
We will explain more details of the holomorphic anomaly equations in Section \ref{BCOV HAE} and the holomorphic limit in Section \ref{Holomorphic Limit}. \\

It is worth mentioning some recent progress on rigorous mathematical studies of $g=1$ mirror symmetry.
The $g=1$ mirror formula \cite{BCOV2} for the quintic Calabi--Yau threefold is first proved in \cite{Zin} and its extension to higher dimension is shown in some cases \cite{Zin2,Pop}.
Inspired by the BCOV theory, the paper \cite{FLY} defines an invariant, called the BCOV torsion, of a one-parameter family of Calabi--Yau threefolds, which is an analogue of the Ray--Singer analytic torsion.
They also identify this invariant is the B-model topological string amplitude for the quintic in \cite{BCOV2}.
%The BCOV torsion for the quintic Calabi--Yau threefold was computed in \cite{FLY}.

%%%%%%%%%%%%%%%%%%%%%%%%%%%%%%%%%%%%%%%%%%%%%%%%%%%%%%%%%%%%%%%%%%%%%%%%%%%%%%%%%%%%%%%%%%%%%%%%%%%%%%%%%%%%%
\section{BCOV Holomorphic Anomaly Equations} \label{BCOV HAE}

The central theme of this section is the BCOV holomorphic anomaly equations \cite{BCOV,BCOV2}, which measure the anti-holomorphicity of the topological string amplitudes $\mathcal{F}_g$ $(g\ge 1)$.
The presence of holomorphic anomaly in the theory makes higher genus mirror symmetry more challenging.

%%%%%%%%%%%%%%%%%%%%%%%%%%%%%%%%%%%%%%%%%%%%%%%%%%%%%%%%%%%%%%%%%%%%%%%%%%%%%%%%%%%%%%

\subsection{Toy Model (Elliptic Curve)} \label{Elliptic curve}
Let us begin our discussion by working on an elliptic curve\footnote{
This case is somewhat misleading because an elliptic curve is a self-mirror manifold.
However, we believe this is still a good example the reader should keep in mind.
}.
We compute the topological string amplitude $F_g(t)$ for an elliptic curve $E$ as a target space.
Since $F_0(t)$ is trivial, the first non-trivial quantity is $F_1(t)$.
The number of connected coverings $E\rightarrow E_\tau$ of degree $d$ is given by the sum of divisors $\sigma(d):=\sum_{k|d}k$
and that each such space is normal with a group of deck transformations of order $d$.
Therefore\footnote{We have to take extra care of the first term of the second line, see \cite{Dij}. }
\begin{align}
F_1(t):&= \int_{\overline{M}_{1,1}}\frac{d\tau d\bar \tau}{(\Im \tau)^2}\sum_{\phi:E\rightarrow E_\tau}e^{2 \pi t\sqrt{-1} \int \phi^*\kappa} \notag \\
&= -\frac{2\pi i t}{24}+\sum_{d>0} \frac{\sigma(d)}{d} e^{2\pi \sqrt{-1}dt} \notag \\
&=-\log(\eta(t)), \notag
\end{align}
where $\eta(t)=q^{\frac{1}{24}}\prod_{n>0}(1-q^n)$ is the Dedekind eta function with $q=e^{2\pi \sqrt{-1} t}$.
The function $F_1(t)$ is unfortunately not modular and we introduce the following non-holomorphic modular function
$$
\mathcal{F}_1(t,\bar t):=-\log (\sqrt{\Im(t)} {\bar \eta}(t) \eta(t)).
$$
This is an example of {\it holomorphic anomaly} and the {\it holomorphic anomaly equation} in this case reads
$$
\partial_t \partial_{\bar t}\mathcal{F}_1(t,\bar t)=\frac{1}{2(t-\bar t)^2}.
$$
This equation together with the modular property recovers the quantity $\mathcal{F}_1(t,\bar t)$. 
It also shows that the holomorphic anomaly is captured by the Poincar\'e geometry. 
For $g\ge 2$, we count the number of coverings of an elliptic curve $E_\tau$ simply ramified at $2g-2$ distinct points.
This number is known as the Hurwitz number.
In \cite{Dij} Dijkgraaf observed that, the topological string amplitude $F_{g}(t)$, the generating function of the Hurwitz numbers, is {\it quasi-modular}.
This is understood as the modular anomaly of $F_g(t)$ for $g\ge 2$.
Let us recall some basics of quasi-modular forms \cite{KZ}.
It is known that the ring of the modular forms is generated by the Eisenstein series $E_4(t),E_6(t)$ over $\mathbb{C}$.
On the other hand, $E_2(t)$ is not modular, but quasi-modular in the sense that
$$
E_2(
\begin{bmatrix}
        a & b\\
        c & d\\
        \end{bmatrix}\cdot t
)
=(ct+d)^2E_2(t)+\frac{6}{\pi \sqrt{-1}} c(ct+d)
$$
and the ring of quasi-modular forms is given by $\mathbb{C}[E_2,E_4,E_6]$. By introducing non-holomorphicity to $E_2(t)$ by
$$
E_2^*(t,\bar t):=E_2(t)+\frac{6}{\pi \sqrt{-1}}\frac{1}{t-\bar t},
$$
we can check that the new function $E_2^*(t,\bar t)$ on $\mathbb{H}$ is modular in a natural sense and thus called an almost-holomorphic modular form (Section \ref{Kaneko--Zagier}).
The ring of almost-holomorphic modular forms is given by $\mathbb{C}[E_2^*,E_4,E_6]$
and there exists a natural object $\mathcal{F}_g \in \mathbb{C}[E_2^*,E_4,E_6]$ associated to $F_g(t)$ for $g \ge 2$.
There is, however, no known explicit holomorphic anomaly equations of higher genus for elliptic curves.

%%%%%%%%%%%%%%%%%%%%%%%%%%%%%%%%%%%%%%%%%%%%%%%%%%%%%%%%%%%%%%%%%%%%%%%%%%%%%%%%%%%%%%

\subsection{Holomorphic Anomaly Equations}
Let $(X^\vee, X)$ be a mirror pair of Calabi--Yau threefolds and $\mathcal{L}$ be the vacuum bundle of the complex moduli space of $X$.
In \cite{BCOV,BCOV2}, Bershadsky, Cecotti, Ooguri and Vafa identified the higher genus topological string amplitude $\mathcal{F}_g$ with $g\geq 2$
as a smooth section of the line bundle $\mathcal L^{2-2g}$ with holomorphic anomaly described by
\begin{equation}
\partial_{\bar i}\mathcal{F}_{g}
=\frac{1}{2}C^{jk}_{\bar i}
( D_j D_k \mathcal{F}_{g-1} + \sum_{r=1}^{g-1} D_j \mathcal{F}_{r} \,D_k \mathcal{F}_{g-r}) \ \ \ (g \ge 2). \label{BCOV HAE g>1}
\end{equation}
The recursive equation (\ref{BCOV HAE g>1}) is called the {\it BCOV holomorphic anomaly equation} (BCOV HAE).
The first term represents the degeneration of a genus $g$ curve to a genus $g-1$ curve.
and the second term represents the degeneration of a genus $g$ curve to genus $r$ and $g-r$ curves (see Fig. 1).
\begin{figure}[h!]
 \begin{center}
  \includegraphics[width=85mm]{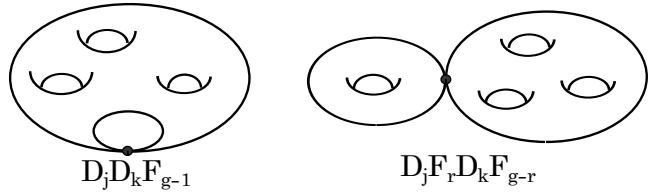}
 \end{center}
 \caption{Degenerating Riemann surfaces contributing to the holomorphic anomaly}
\end{figure}

For $g=1$, the holomorphic anomaly of the topological string amplitude $\mathcal{F}_1$  is measured by the following:
\begin{equation} \label{BCOV HAE g=1}
\partial_i\partial_{\bar j}\mathcal{F}_{1} =\frac{1}{2}C_{ijk}C^{kl}_{\bar j}+(1-\frac{\chi(X^\vee)}{24})G_{i \bar j}.
\end{equation}
This is known as the $tt^{*}$-equation.
In \cite{BCOV,BCOV2} they also conjectured that the smooth function $\mathcal{F}_{1}$ is obtained as the Ray--Singer torsion.
For $g \ge 2$, there is no easy mathematical definition of topological string amplitudes $\mathcal{F}_g \in \Gamma_{C^\infty}(\mathcal{M},\mathcal L^{2-2g})$,
and thus we define them as solutions to the BCOV holomorphic anomaly equations in (\ref{BCOV HAE g>1}) with certain boundary conditions.\\

The basic idea for solving the equation (\ref{BCOV HAE}) is to re-express RHS of the equation as anti-holomorphic derivatives so that we can integrate them up to some holomorphic ambiguity.
For example, in the case where $h^{2,1}(X)=1$, the $tt^{*}$-equation reads
$$
\partial_z \partial_{\bar z}\mathcal{F}_1(z,\bar z) = \frac{1}{2}C_{zzz}C_{\bar z\bar z\bar z}e^{2K}G^{z\bar z}G^{z\bar z}-(\frac{\chi(X^\vee)}{24}-1)G_{z\bar z}.
$$
A solution of the $tt^{*}$-equation is explicitly given by
\begin{equation}\label{soltogenus1hae}
\mathcal{F}_1(z,\bar z)=\frac{1}{2}\log(G^{z \bar z}e^{K(4-\frac{\chi(X^\vee)}{12})})+{1\over 2}|f_1(z)|^{2},
\end{equation}
for some holomorphic ambiguity $f_1(z)$ because
\begin{align}
\partial_z \partial_{\bar z} \mathcal{F}_1 (z,\bar z)
&=\frac{1}{2}\partial_{\bar z}(-\partial_z \log G_{z \bar z} +(4-\frac{\chi(X^\vee)}{12})K_z)      \notag \\
&=\frac{1}{2}(-\partial_{\bar z}\Gamma_{zz}^{\;\;z}+(4-\frac{\chi(X^\vee)}{12})G_{z \bar z})     \notag \\
&=\frac{1}{2}\partial_{\bar z}(C_{zzz}\overline{C}^{zz}_{\bar z}-(\frac{\chi(X^\vee)}{24}-1)G_{z \bar z}). \notag
\end{align}
In the last line we used the special K\"ahler geometry relation (\ref{Special Kahler geometry relation}).

\subsection{Propagators and Polynomiality}
Solving the BCOV holomorphic anomaly equation for large $g$ is very involved
and we need to make the use of certain polynomiality of topological string amplitudes.
In \cite{BCOV2} the authors found it convenient to introduce the following {\it propagators} $S,S^i,S^{ij}$:
$$
S\in \Gamma(\mathcal{M},\mathcal{L}^{-2}), \ \ \
S^i\in \Gamma(\mathcal{M},\mathcal{L}^{-2}\otimes T\mathcal{M}),\ \ \
S^{ij}\in \Gamma(\mathcal{M},\mathcal{L}^{-2}\otimes \mathrm{Sym}^2(T\mathcal{M})),
$$
with relations
\begin{equation}\label{defofpropagators}
C_{\bar i \bar j\bar k}=e^{-2K}D_{\bar i}D_{\bar j} \partial_{\bar k}S, \ \ \
\partial_{\bar i}S^{ij}=C^{ij}_{\bar i}, \ \ \
\partial_{\bar i}S^j=G_{k\bar i}S^{kj}, \ \ \
\partial_{\bar i}S=G_{j \bar i}S^j.
\end{equation}
As the name suggests, they make the connection to the Feynman diagram interpretation in \cite{BCOV2} clearer.
Although the general solutions of the BCOV holomorphic anomaly equations can be obtained by the standard Feynman rules,
for higher genus the number of diagrams grows very quickly with the genus.
\begin{example}
The topological string amplitude $\mathcal{F}_2 (z,\bar z)$ is written as
\begin{eqnarray}\label{soltogenus2hae}
\mathcal{F}_2(z,\bar z)&=& {1 \over 2} S^{ij} D_{ij}\mathcal{F}_1 + {1 \over 2} S^{ij} D_i\mathcal{F}_1 D_j\mathcal{F}_1 - {1 \over 8} S^{jk}S^{mn}D_{jkmn}\mathcal{F}_0 \notag\\
&&-{1 \over 2}S^{ij}S^{mn} D_{ijm}\mathcal{F}_0D_n\mathcal{F}_1 +{\chi \over 24} S^i D_i\mathcal{F}_1  \notag\\
&&+{1 \over 8} S^{ij} S^{pq} S^{mn} D_{ijp}\mathcal{F}_0 D_{qmn} \mathcal{F}_0 + {1 \over 12}S^{ij}S^{pq}S^{mn}D_{ipm}\mathcal{F}_0D_{jqn}\mathcal{F}_0 \notag \\
&&- {\chi \over 48} S^i S^{jk}D_{ijk}\mathcal{F} _0 +{\chi(X^\vee) \over 24}({\chi(X^\vee) \over 24}-1) S + f_2(z),
\end{eqnarray}
where $D_{i_1\dots i_k}:=D_{i_1}\dots D_{i_k}$ and $f_2(z)$ represents a holomorphic ambiguity.
\end{example}
\begin{example}
The topological string amplitude $\mathcal{F}_3$ is written as
\begin{align}
\mathcal{F}_3=& {1 \over 2} S^{ij} D_{ij}\mathcal{F}_{2} + D_{i}\mathcal{F}_{1} S^{ij} D_{j}\mathcal{F}_{2} +({\chi \over 24} + 2) S^{i} D_{i}\mathcal{F}_{2} \notag\\
& + 2 \mathcal{F}_2 S^i D_{i}\mathcal{F}_{1} -{1 \over 2} S^{ij} D_{ijk}\mathcal{F}_{0} \,S^{kl} D_{l}\mathcal{F}_{2}- {1 \over 4} S^{ij}S^{kl} D_{ijkl}\mathcal{F}_{1} \notag\\
& - {1 \over 2} S^{ij} D_{ijk}\mathcal{F}_{1}\,S^{kl} D_{l}\mathcal{F}_{1} -{1 \over 4} S^{ij}S^{kl} D_{ik}\mathcal{F}_{1} D_{jl} \mathcal{F}_{1} + \cdots + f_3(z),\notag
\end{align}
where $f_3(z)$ represents a holomorphic ambiguity.
\end{example}

Motivated by the work \cite{BCOV2}, in \cite{YY} Yamaguchi and Yau show
for the mirror quintic family that the topological string amplitudes $\mathcal{F}_g$ are polynomials in the propagators $S^{ij},S^i,S$ and the K\"ahler derivatives $K_{i}$.
This was generalized in \cite{AL} to general Calabi--Yau threefolds.
The polynomiality for the topological string amplitudes $\mathcal{F}_g$ provides a significant enhancement for
practical computations and also equips the ring generated by the propagators and K\"ahler derivatives with interesting mathematical structures.
A more detailed overview of this subject, as well as the connection of the ring to modular forms \cite{ABK, Hos, ASYZ, Zho, Ali2}, can be found in a separate expository article \cite{Zho2}.

%%%%%%%%%%%%%%%%%%%%%%%%%%%%%%%%%%%%%%%%%%%%%%%%%%%%%%%%%%%%%%%%%%%%%
%%%%%%%%%%%%%%%%%%%%%%%%%%%%%%%%%%%%%%%%%%%%%%%%%%%%%%%%%%%%%%%%%%%%%

\section{Holomorphic Limits and Boundary Conditions} \label{Holomorphic Limit}
In this section we first discuss {\it holomorphic limits}, which relate an almost-holomorphic object $\mathcal{F}_g \in \Gamma_{C^\infty}(\mathcal{M},\mathcal{L}^{2-2g})$ to
a holomorphic object $F_g \in \Gamma(\mathcal{M},\mathcal{L}^{2-2g})$.
We then turn to the boundary conditions of the topological string amplitudes $\mathcal{F}_g$.
The holomorphic limit and boundary conditions should be compared with the theory of (quasi- and almost-holomorphic) modular forms \cite{ABK, ASYZ}.

%%%%%%%%%%%%%%%%%%%%%%%%%%%%%%%%%%%%%%%%%%%%%%%%%%%%%%%%%%%%%%%%%%%%%

\subsection{Toy Model (Kaneko--Zagier Theory)} \label{Kaneko--Zagier}
It is instructive to compare the holomorphic limit with the classical theory of modular forms (see also Section \ref{Elliptic curve}).
We briefly review the Kaneko--Zagier theory \cite{KZ}.
We consider the almost-holomorphic modular forms $\widehat M(\Gamma)_k$
of weight $k$ as the functions $F(t,\bar t) \in \mathbb{C}[[t]] \big[ \frac{1}{t-\bar t} \big]$ on $\mathbb{H}$ which transforms just like a modular form of weight $k$;
$$
F(
\begin{bmatrix}
        a & b\\
        c & d\\
        \end{bmatrix}\cdot t,
\overline{
\begin{bmatrix}
        a & b\\
        c & d\\
        \end{bmatrix}} \cdot \bar t
)
=(ct+d)^kF(t,\bar t).
$$
The ring of the almost-holomorphic modular forms $\widehat M(\Gamma) :=\oplus_{k \ge 0} \widehat M(\Gamma)_k$
is given by $\widehat M(\Gamma) =\mathbb{C}[E_2^*,E_4,E_6]$ and becomes a differential ring under the operator
$$
\frac{1}{2\pi i} \big( \frac{\partial}{\partial t}+\frac{k}{t-\bar t} \big):
\widehat M(\Gamma)_k \rightarrow \widehat M(\Gamma)_{k+2}.
$$
The elements of $\widehat M(\Gamma)$ have an expansion of the form
$F(t,\bar t)=\sum_{m \ge 0} \frac{F_m(t)} {(t-\bar t)^m}$.
The key observation \cite{KZ} is that the map
$$
\phi: \widehat M(\Gamma) \rightarrow \mathbb{C}[E_2(t),E_4(t),E_6(t)], \ \ \ F(t,\bar t) \mapsto F_0(t).
$$
is a differential ring isomorphism,
where the LHS is equipped with the differential $\frac{1}{2\pi \sqrt{-1}} \frac{\partial}{\partial t}$.
As we mentioned earlier, the map $\phi$ gives a correspondence between $\mathcal{F}_g$ and $F_g$ for the elliptic curves.
We observe that these rings are governed by the Poincar\'e metric
$ ds^2:=-\partial_t \partial_{\bar t} \log(t-\bar t)$ on $\mathbb{H}$.
We can think of the Weil--Petersson metric and the holomorphic limit as higher dimensional analogues of the Poincar\'e metric and the map $\phi$ respectively.
This similarity has been further analyzed in \cite{ABK, Hos,Zho}.

%%%%%%%%%%%%%%%%%%%%%%%%%%%%%%%%%%%%%%%%%%%%%%%%%%%%%%%%%%%%%%%%%%%%%%%%%%%%%%%%%%%%%%

\subsection{K\"ahler Normal Coordinates}  \label{Kahler normal coordinates}
Let $M$ be a K\"ahler manifold of dimension $m$ with  with K\"ahler potential $K(z,\bar{z})$.
The canonical coordinates $\{t^i\}_{i=1}^m$ around $p=(a,\bar{a}) \in M$
are defined to be the holomorphic coordinates such that
\begin{equation}\label{canonicalcoordinates}
\partial_{I} K_{i}|_{p}=0=\partial_{I}\Gamma_{ij}^{k}|_{p},
\end{equation}
where $\partial_{I}=\partial_{t^{i_{1}}}\cdots\partial_{t^{i_{n}}}$ for $I=(i_{1},i_{2},\cdots i_{n})$.
One can locally solve the second equation in (\ref{canonicalcoordinates}) for $t$ to get the following, see e.g. \cite{HIN,GS}:
\begin{equation*}\label{solofcanonicalcoordinates}
t^{i}(z)=K^{i\bar{j}}(a,\bar{a})(K_{\bar{j}}
(z,\bar{a})-K_{\bar{j}}(a,\bar{a}))\,.
\end{equation*}
The holomorphic function $f(z,\bar{a})$ is the degree $0$ part
in the Taylor expansion of the function $f(z,\bar{z})$ in $\bar{z}$ centered at $\bar{a}$.
This will be explained below using a holomorphic exponential map \cite{Kap}.\\

We first consider the exponential map $\exp^{\mathbb{R}}_{p}: T^{\mathbb{R}}_{p}M\rightarrow M$ as a Riemannian manifold.
Thinking of $T^{\mathbb{R}}_{p}M$ as a complex vector space equipped with the complex structure induced by that on $M$,
the map $\exp_{p}^{\mathbb{R}}: (\xi,\bar{\xi})\mapsto (z(\xi,\bar{\xi}),\bar{z}(\xi,\bar{\xi}))$ is in general not holomorphic.
Now with the assumption that the metric $G_{i\bar{j}}(z,\bar{z})$ is analytic in $z,\bar{z}$,
we can analytically continue the map $\exp_{p}^{\mathbb{R}}$ to the corresponding complexifications $T^{\mathbb{C}}_{p}M$ and $M_{\mathbb{C}}=M\times \overline{M}$,
where $\overline{M}$ is the complex manifold with complex structure opposite to that on $M$.\\

The coordinates on the complexifications $T^{\mathbb{C}}_{p}M$ and $M_{\mathbb{C}}=M\times \overline{M}$ are respectively given by $(\xi,\eta)$ and $(z,w)$,
which are the analytic continuation of the coordinates $(\xi,\bar{\xi})$ and $(z,\bar{z})$ from
$T^{\mathbb{R}}_{p}M\hookrightarrow T^{\mathbb{C}}_{p}M$ and $\Delta:M\hookrightarrow M_{\mathbb{C}}=M\times \overline{M}$ respectively.
Here $\Delta: M\rightarrow M\times \overline{M}, p\mapsto (p,\bar{p})$ is the diagonal embedding.
The underlying point of $\bar{p}$ is the same as $p$, but we have used the barred notation to indicate that it is a point on $\overline{M}$.\\

Since the Christoffel symbols $\Gamma_{ij}^{k}(z,\bar{z})$ are analytic in $(z,\bar{z})$,
we know that the map $\exp_{p}^{\mathbb{C}}: (\xi,\eta)\mapsto (z(\xi,\eta),w(\xi,\eta))$ is analytic, that is, holomorphic in $(\xi,\eta)$.
Moreover, the map $\exp_{p}^{\mathbb{C}}$ defines a local bi-holomorphism from a neighbourhood around $0 \in T_{p}^{\mathbb{C}}M$
to a neighbourhood of $(p,\bar{p}) \in M_{\mathbb{C}}$.
One claims that $\exp^{\mathbb{C}}_{p}|_{T^{1,0}M}$ gives a holomorphic map $T^{1,0}_{p}M\rightarrow M$
which is locally bi-holomorphic near $0\in T_{p}^{1,0}M$.
To show that it maps $T^{1,0}_{p}M$ to $M$, it suffices to show that $w\circ \exp^{\mathbb{C}}_{p}|_{T^{1,0}_{p}M}=w(\bar{p})$, that is, $w(\xi,0)=w(\bar{p})$.
Recall that $\bar{z}$ and thus $w$ satisfies the equation for the geodesic equation
\begin{equation*}
{d^{2}\over
ds^{2}}\bar{z}^{k}+\Gamma^{\bar{k}}_{\bar{i}\bar{j}}{d\bar{z}^{\bar{i}}\over ds}{d\bar{z}^{\bar{j}} \over ds}=0,\ \ \
 {d\bar{z}^{\bar{k}}\over ds}(0)=\bar{\xi}^{\bar{k}}=0,\ \ \bar{z}(0)=\bar{z}(\bar{p})\,.
\end{equation*}
It is easy to see that $w(s)=w(\bar{p})$ is one and thus the unique solution to the differential equation.
Therefore, $w\circ \exp^{\mathbb{C}}_{p}(\xi,0)=w(\bar{p})$ as desired.
Since $z(\xi,\eta)$ is holomorphic in both $\xi,\eta$, we know $z(\xi,0)$ is holomorphic in $\xi$. The same reasoning for the
exponential map $\exp_{p}^{\mathbb{R}}$ shows that it is locally a bi-holomorphism.\\

Hence one gets a holomorphic exponential map $\exp^{\textrm{hol}}_{p}=\exp^{\mathbb{C}}_{p}|_{T^{1,0}M}: T^{1,0}_{p}M\rightarrow M$.
We now denote the coordinate $\xi$ on $T^{1,0}_{p}M$ by $t$, and then this is the canonical coordinates desired.
The exponential maps $\exp^{\mathbb{R}}_{p}$ and $\exp^{\mathrm{hol}}_{p}$ are contrasted as follows:
\begin{eqnarray*}
\exp_{p}^{\mathbb{R}}&=&\exp^{\mathbb{C}}_{p}|_{T^{\mathbb{R}}_{p}M}=\exp^{\mathbb{C}}_{p}|_{T^{1,0}_{p}M\oplus
\overline{T_{p}^{1,0}M}}\,,\\
\exp_{p}^{\textrm{hol}}&=&\exp^{\mathbb{C}}_{p}|_{T^{1,0}_{p}M}=\exp^{\mathbb{C}}_{p}|_{j(T^{1,0}_{p}M)=T^{1,0}_{p}M\oplus
\{0\}}\,.
\end{eqnarray*}
where $T^{1,0}_{p}M\oplus \overline{T^{1,0}_{p}M}$ means the image of the map
$ T^{1,0}_{p}M\rightarrow T^{1,0}_{p}M\oplus T^{0,1}_{p}M,\, v\mapsto (v,v^{*})$,
where $v^{*}$ is the complex conjugate of $v$; and $j(T^{1,0}_{p}M)$ is the image of the map
$j: T^{1,0}_{p}M\mapsto T^{1,0}_{p}M\oplus T^{0,1}_{p}M,\, v\mapsto (v,0)$.

%%%%%%%%%%%%%%%%%%%%%%%%%%%%%%%%%%%%%%%%%%%%%%%%%%%%%%%%%%%%%%%%%%%%%%%%%%%%%%%%%%%%%%

\subsection{Examples of Canonical Coordinates}
In this section we shall compute the canonical coordinates for some examples of K\"ahler manifolds.

\begin{example}[Fubini--Study metric]
Consider the Fubini--Study metric on $\mathbb{P}^{1}$ with K\"ahler potential $K=\log (1+|z|^{2})$. It follows then
\begin{equation*}
K_{z}={\bar{z}\over (1+|z|^{2})},\quad K_{z\bar{z}}={1\over
(1+|z|^{2})^{2}}, \quad\partial_{z}^{N}K_{\bar{z}}={(-1)^{N+1}N!
\bar{z}^{N-1}\over (1+|z|^{2})^{N+1}},\,N\geq 1\,.
\end{equation*}
We see that $z$ is the canonical coordinate based at $a=0$.
To find the canonical coordinate at a point $p$ represented by $a\neq 0$,
we apply Eq. (\ref{solofcanonicalcoordinates}) and get
\begin{equation*}
t(z)=(1+|a|^{2})^{2}\left({z\over
(1+z\bar{a})}-{a\over (1+a\bar{a})}\right)\,.
\end{equation*}
We see that the canonical coordinates have non-holomorphic dependence on the base-point.
\end{example}

\begin{example}[Poincar\'e metric]  \label{Poincare metric}
Consider the $\textrm{SL}(2,\mathbb{Z})$--invariant metric on $\mathbb{H}$
\begin{equation*}
\omega={\sqrt{-1}\over 2} K_{\tau\bar{\tau}}d\tau\wedge d\bar{\tau}={1\over y^{2}} dx\wedge dy,
\end{equation*}
where $e^{-K}={\tau-\bar{\tau}\over \sqrt{-1}}$, $\tau=x+\sqrt{-1}y$.
Straightforward computations show that
\begin{equation*}
K_{\bar{\tau}}={1\over \tau-\bar{\tau}}, \ \ \
K_{\tau\bar{\tau}}=-{1\over (\tau-\bar{\tau})^{2}}.
\end{equation*}
It follows that the canonical coordinate based at $p$ given by $a$ is
\begin{equation*}
t(\tau)=-(a-\bar{a})^{2}\left(
{1\over \tau-\bar{a}}-{1\over a-\bar{a}}\right).
\end{equation*}
For $a=i\infty$, the canonical coordinate $t$ coincides with the complex coordinate on $\mathbb{H} \subset \mathbb{C}$.
\end{example}

\begin{example}[Weil--Petersson metric for elliptic curve family]\label{WPforellipticcurve}
Consider the ``universal" elliptic curve family parametrized by $\mathbb{H}$.
Fixing the holomorphic top form $\Omega_{\tau}=dz_{\tau}=dx+\tau dy$ on $T_{\tau}$.
Using the diffeomorphism from the fiber $T_{\tau}$ to the fiber $T_{a}$ given by
\begin{equation*}
z_{\tau}={\tau-\bar{a}\over
a-\bar{a}}z_{a}+{a-\tau\over
a-\bar{a}}\bar{z}_{a}\,,
\end{equation*}
we can compute the K\"ahler potential for the Weil-Peterson metric from
\begin{equation*}
e^{-K(\tau,\bar{\tau})}= \sqrt{-1} \int_{T_{\tau}}\Omega_{\tau}\wedge
\overline{\Omega}_{\tau}={\tau-\bar{\tau}\over \sqrt{-1}}.
\end{equation*}
This is  the Poincare metric considered in Example \ref{Poincare metric}.
\end{example}

\begin{example}
Let $M$ be a K\"ahler manifold with local coordinates $\{z^{i}\}$
and a holomorphic function $F(z)$ such that the K\"ahler potential $K$ is given by $K={1\over 2}~\mathrm{Im} w_{i}\bar{z}^{i}$ where $w_{i}(z)={\partial_{i}F(z)}$.
A K\"ahler manifold of this type is a special K\"ahler manifold \cite{Fre} and the canonical coordinates are then given by
\begin{equation*}
t^{i}(z)={1\over
\tau_{ij}(a)-\bar{\tau}_{ij}(\bar{a})}(w_{j}(z)-w_{j}(a)-
\bar{\tau}_{jk}(\bar{a})(z^{k}-a^{k})),
\end{equation*}
where $ \tau_{ij}(z)=\partial_{i}\partial_{j}F(z)$.
\end{example}

%%%%%%%%%%%%%%%%%%%%%%%%%%%%%%%%%%%%%%%%%%%%%%%%%%%%%%%%%%%%%%%%%%%%%%%%%%%%%%%%%%%%%%

\subsection{Holomorphic Limits}\label{sectionhollimits}

The holomorphic limit of a function $f(z,\bar{z})$ based at $a$ is defined as follows.
First one analytically continues the map $f$ to a map defined on $M_{\mathbb{C}}$.
Using the fact that $\exp^{\mathbb{C}}_{p}$ is a local diffeomorphism from $T_{p}^{\mathbb{C}}M$ to $M_{\mathbb{C}}$,
we get $\hat{f}=f\circ \exp_{p}^{\mathbb{C}}: T_{p}^{\mathbb{C}}M\rightarrow \mathbb{C}$.
The holomorphic limit of $f(z,\bar{z})$ is given by $\hat{f}|_{j(T^{1,0})}:T^{1,0}_{p}M\rightarrow T_{p}^{\mathbb{C}}M\rightarrow \mathbb{C}$.
The coordinates $(z,\bar{z})$ and $(t,\bar{t})$ are often used for $(z,w)$ and $(\xi,\eta)$ when considering holomorphic limits.\\

In the canonical coordinates $t$ on the K\"ahler manifold $M$,
the holomorphic limit of $f$ based at $a$ is described by
$$
f\circ \exp^{\textrm{hol}}_{a}=\hat{f}|_{j(T^{1,0}_{a}M)}: T^{1,0}_{a}M\times \{0\}\rightarrow \mathbb{C}, \ \ \ t\mapsto f\circ \exp^{\textrm{hol}}_{a}(t).
$$
In terms of an arbitrary local coordinate system $z$ on $M$,
taking the holomorphic limit of the function $f(z,\bar{z})$ at the base point $a$ is
the same as keeping the degree zero part of the Taylor expansion of $f(z,\bar{z})$ with respect to $\bar{z}$.\\

Let us return to the special K\"ahler geometry on the moduli space of complex structures of a Calabi--Yau threefold.
It can be shown that the special coordinates $\{t^i\}_{i=1}^n$ defined near a LCSL are the canonical coordinates \cite{BCOV}.
Moreover, rewriting the defining equation for the K\"ahler potential introduced in Section \ref{Special Kahler Manifolds} as
\begin{equation*}
e^{-K(z,\bar{z})}=\phi^{0}\overline{\phi^{0}}e^{-K(t,\bar{t})},\quad
e^{-K(t,\bar{t})}=\sqrt{-1}\left(2\overline{F(t)}-2F(t)+(t^{a}-\bar{t}^{a})(F_{a}+\overline{F_{a}})\right),
\end{equation*}
we obtain
\begin{equation*}
K_{i}=-\partial_{i}\log \phi^{0}+K_{a}{\partial t^{a}\over
\partial z^{i}},\quad \Gamma_{ij}^{k}={\partial z^{k}\over \partial t^{a}}{\partial \over \partial z^{i}} {\partial t^{a}\over
\partial z^{j}}+{\partial z^{k}\over \partial t^{c}}\mathrm{\Gamma}_{ab}^{c}{\partial t^{a}\over \partial z^{j}}
{\partial t^{b}\over \partial z^{j}}.
\end{equation*}
Then, according to (\ref{canonicalcoordinates}), their holomorphic limits at the LCSL  are given by:
\begin{equation}\label{hollimitsofconnections}
\lim_{\bar t \to \sqrt{-1}\infty} K_{i}=-\partial_{i}\log \phi^{0},\ \ \
\lim_{\bar t \to \sqrt{-1}\infty} \Gamma_{ij}^{k}={\partial z^{k}\over \partial t^{a}}{\partial \over \partial z^{i}} {\partial t^{a}\over \partial z^{j}}.
\end{equation}
We used the notation $\lim_{\bar t \to \sqrt{-1}\infty}$ because the LCSL corresponds to $\sqrt{-1}\infty$ in the mirror coordinates $t$. In the rest of the article, we shall use the notation $\lim_{a}$ to denote the holomorphic limit based at the point $a$.

\subsection{Boundary Conditions}

As we have mentioned in Section \ref{BCOV HAE}, the holomorphic anomaly equations only determine
the topological string amplitude $\mathcal{F}_{g}$ up to some holomorphic ambiguity $f_{g}(z)$
and certain boundary conditions on the moduli space $\mathcal{M}$ are needed to fix the ambiguity $f_{g}(z)$.
What are commonly used are the physical interpretation of the asymptotic behaviors of $\mathcal{F}_{g}$ at the singular points on the moduli space $\mathcal{M}$.
The boundary conditions of $\mathcal{F}_{g}$ at the LCSL (mirror to the large volume limit of $X^\vee$ given by $t^{i}=\sqrt{-1}\infty$ for $1\le i \le h^{1,1}(X^\vee)$)
and at the conifold loci are satisfied by the holomorphic limits of the normalized topological string amplitude $(\phi^0)^{2g-2}\mathcal{F}_{g}$
based at the corresponding loci on the moduli space \cite{BCOV,BCOV2,Gho, Ant}. \\

At the LCSL, the boundary conditions read
\begin{eqnarray}\label{bdyconditionatLCSL}
\lim_{\textrm{LCSL}} \mathcal{F}_{1}&=&-{1\over 24}t^{i}\int_{X^\vee}c_{2}(X^\vee) \cup T_{i} +O(e^{2\pi i t})\nonumber\,,\\
\lim_{\textrm{LCSL}} (\phi^0)^{2g-2} \mathcal{F}_{g}&=&(-1)^{g}{\chi(X^\vee)\over 2}{|B_{2g}B_{2g-2}|\over 2g(2g-2)(2g-2)!}+O(e^{2\pi i t}),\quad g\geq 2.
\end{eqnarray}
Of course, these come from the expression of the Gromov--Witten potentials in (\ref{AmodelFg}).
The boundary conditions at the conifold locus (CON) determined by $\Delta_{j}(z)=0$ $(1\le j \le m)$ read
\begin{eqnarray}\label{bdyconditionatCON}
\lim_{\textrm{CON}} \mathcal{F}_{1}&=&-{1\over 12}\log t_{c}^{j} +\textrm{regular~function}\nonumber\,,\\
\lim_{\textrm{CON}} (\phi_{\textrm{CON},j}^0)^{2g-2} \mathcal{F}_{g}&=&{(c_j)^{g-1}|B_{2g}|\over 2g(2g-2)(t_{\textrm{CON}, j})^{2g-2}}+\textrm{regular~function},\quad g\geq 2 \label{gap}\,,
\end{eqnarray}
where
$\phi^{0}_{\textrm{CON},j}$ and $t_{\textrm{CON}, j}=\phi^{1}_{\textrm{CON},j}/\phi^{0}_{\textrm{CON},j}$ are the regular period and the normalized vanishing period $\phi^{1}_{c,j}$ near the conifold locus $\Delta_{j}=0$ respectively,
and $c_j$ is a constant independent of genus $g$.
The condition in (\ref{gap}) is often called the gap condition due to the fact that the sub-leading terms are vanishing \cite{HK1, HKQ}.\\

In a good situation and for small $g$, these boundary conditions suffice to determine the holomorphic ambiguity $f_{g}$ and thus $\mathcal{F}_{g}$ to a large extent (see \cite{HKQ,HK} and references therein).

%%%%%%%%%%%%%%%%

\section{Examples}  \label{Examples}
In this section we shall review mirror symmetry of some compact and non-compact Calabi-Yau threefold families.

\subsection{Quintic Threefold}

Consider the Dwork pencil of quintic threefolds for $\psi \in \mathbb{C}$:
\begin{equation*}
X^\vee_{\psi}:=\{x_{1}^{5}+x_{2}^{5}+x_{3}^{5}+x_{4}^{5}+x_{5}^{5}-\psi x_{1}x_{2}x_{3}x_{4}x_{5}=0\} \subset \mathbb{P}^4.
\end{equation*}
The mirror manifold $X_{\psi}$ is obtained as a crepant resolution of the orbifold
$$
X_{\psi}:=X^\vee_{\psi}/G, \ \ \ \psi \in \mathbb{C},
$$
where
$$
G=\left\{(a_i)\in (\mathbb{Z}_5)^5 \ | \ \sum_{i=1}^5a_i=0 \right\}/\mathbb{Z}_5 \cong (\mathbb{Z}_5)^3.
$$
We refer the reader to \cite{Gre,CdOGP} for details.
The Picard--Fuchs equation of the mirror family reads
\begin{equation}
\left(\theta^{4}-5^{5}z (\theta+{1\over 5})(\theta+{2\over 5})(\theta+{3\over 5})(\theta+{4\over 5}) \right)\phi(z)=0, \label{PF eq}
\end{equation}
where $z=(5\psi)^{-5}$ and $\theta=z{\partial\over \partial z}$.
By the Griffiths transversality, we have
$$
z^{3}C_{zzz}=-\int_{X}\Omega\wedge \theta^{3}\Omega.
$$
Again by using the Griffiths transversality and Picard--Fuchs equation (\ref{PF eq}), we obtain
\begin{eqnarray*}
\theta (z^{3}C_{zzz})
&=&-\int_{X}\theta \Omega\wedge \theta^{3}\Omega-\int_{X}\Omega\wedge \theta^{4}\Omega\\
&=&-\theta  \int_{X}\theta \Omega\wedge \theta^{2}\Omega+\int_{X}\theta^{2} \Omega\wedge \theta^{2}\Omega-\int_{X}\Omega\wedge \left({2 \cdot5^{5}z\over 1-5^{5}z} \theta^{3}\Omega+\cdots\right)\\
&=&-\theta\left( \theta \int_{X}\Omega\wedge \theta^{2} \Omega-\int_{X} \Omega\wedge \theta^{3} \Omega \right)\\
&&
+\int_{X}\theta^{2} \Omega\wedge \theta^{2}\Omega-\int_{X}\Omega\wedge \left( {2\cdot 5^{5}z\over 1-5^{5}z} \theta^{3}\Omega+\cdots\right)\\
&=& -\theta (z^{3}C_{zzz}) -{2\cdot 5^{5}z\over 1-5^{5}z}(z^{3}C_{zzz}).
\end{eqnarray*}
Solving for $z^{3}C_{zzz}$ from this first order differential equation, we get
\begin{equation*}
C_{zzz}={c\over z^{3}(1-5^{5}z)}\,,
\end{equation*}
for some constant $c$.
Near the large complex structure limit $z=0$, the special coordinate $t(z)$ is an infinite series in $z$ computed from the periods
$\phi^{0}(z)\sim \mathrm{regular}, \ \phi^{1}(z)\sim \log z+\cdots$.
Mirror symmetry then predicts that under the mirror map $t(z)=\phi^{1}(z)/\phi^{0}(z)$, we should have as in (\ref{normalizedYukawa}):
\begin{equation*}
K_{ttt}=\phi^{0}(z)^{-2} ({\partial z\over \partial t})^{3}{c\over z^{3}(1-5^{5}z)}\,.
\end{equation*}
Comparing the asymptotic behaviors of both sides as $z\rightarrow 0$ or equivalently $t\mapsto \sqrt{-1} \infty$, we find $c=5$.
Thus we can determine the $g=0$ Gromov--Witten invariants by comparing the $q$-series expansions, where $q=e^{2\pi \sqrt{-1} t(z)}$. \\

Genus one mirror symmetry was worked out in \cite{BCOV} by using the holomorphic anomaly equation for $\mathcal{F}_{1}$.
The solution is given by the formula in (\ref{soltogenus1hae})
\begin{equation}\label{F1formirrorquintic}
\mathcal{F}_{1}={1\over 2}\log (G^{z\bar{z}}e^{K(4-{\chi(X^\vee)\over 12})})+ {1\over 2}|\log z^{b}(1-5^{5}z)^{a}|^{2}\,,
\end{equation}
for some constants $a,b$.
To fix these constants, we use the boundary conditions for $\mathcal{F}_{1}$ at the LCSL $z=0$ and at the conifold point $z=1/5^{5}$.
The latter implies that $a=-1/6$.
The former says that in the holomorphic limit at the LCSL,
from (\ref{bdyconditionatLCSL}) we obtain
\begin{equation}\label{hollimitatLCSL}
\lim_{\mathrm{LCSL}}\mathcal{F}_1=-\frac{t}{24} \int_{X^\vee}c_{2}(X^\vee)\cup H+O(e^{2\pi \sqrt{-1} t})\,,
\end{equation}
where $H$ is the hyperplane class of $X^\vee$.
To compute the holomorphic limit of the quantities involved in $\mathcal{F}_{1}$
at the large complex structure limit, we use the results discussed in Section \ref{sectionhollimits}.
According to the asymptotic behaviors
\begin{equation*}
\phi^{0}(z)=1+O(z),\quad t(z)={\phi^{1}(z)\over \phi^{0}(z)}=\log z+O(z)\,.
\end{equation*}
and using the formulas in (\ref{hollimitsofconnections}) we get the following asymptotic behaviors of the holomorphic limits
\begin{equation*}
G_{z\bar{z}}\sim {\partial t\over \partial z}\sim {1\over z}, \quad K_{z}\sim -\partial_{z}\log \phi^{0}(z)=\mathrm{regular~ function}\,.
\end{equation*}
Comparing the asymptotic behaviors of both sides in (\ref{hollimitatLCSL}), we get
\begin{equation*}\label{eqnforamirrorquintic}
{1\over 2}+{b\over 2}=-{1\over 24}\int_{X^\vee}c_{2}(X^\vee)\cup H\,.
\end{equation*}
In the current case, we have $\chi(X^\vee)=-200$ and $\int_{X^\vee}c_{2}(X^\vee)\cup H=50$ and thus we get the full solution
\begin{equation*}
\mathcal{F}_1={1\over 2}\log (G^{z\bar{z}}e^{{62\over 3}K})+ {1\over 2}|\log z^{-{31\over 6}}(1-5^{5}z)^{-{1\over 6}}|^{2}.
\end{equation*}
By using the mirror map and the holomorphic limit for $G_{z\bar{z}}$, we can then write the holomorphic limit\footnote{Computationally, for genus one amplitude,
we need to take its derivative to get rid of the anti-holomorphic terms. Also the generating function of genus one Gromov-Witten invariants with one insertion, which is given by the first derivative of $\mathcal{F}_{1}$, is more natural due to stability reasons.}:
\begin{equation*}
\partial_{t}{\tt F}_{1}(t)=-{1\over 2}\partial_{t}\log {\partial t\over \partial z}
-{31\over 3}\partial_{t}\log \phi^{0}(z)+ {1\over 2}\partial_{t}\log z^{-{31\over 6}}(1-5^{5}z)^{-{1\over 6}}.
\end{equation*}
We refer the reader to \cite{FLY, Zin} for mathematical proofs of this formula.
Comparing it with the expected form obtained from (\ref{AmodelFg}), we get the $g=1$ Gromov--Witten invariants.\\

Genus two case is much more involved than the above two case, but was worked out in \cite{BCOV2}.
The result is given by the formula in (\ref{soltogenus2hae}), with the holomorphic ambiguity $f_2(z)$
\begin{equation*}
f_2(z)=-{71375\over 288}-{10375\over 288}{1\over (1-5^{5}z)}
+{625\over 48}{1\over (1-5^{5}z)^{2}}\,.
\end{equation*}
The propagators $S^{ij},S^{i},S$ can be solved explicitly from the equations in (\ref{defofpropagators}) that they satisfy \cite{BCOV2}.
This determines the $g=2$ Gromov--Witten invariants in the same manner as above.\\

For higher genus $\mathcal{F}_{g}$, the non-holomorphic part is a polynomial in the propagators and the K\"ahler derivatives
which can be solved genus by genus recursively \cite{YY,AL}, as mentioned in Section 4.3.
The holomorphic ambiguities can be fixed by using the boundary conditions up to genus $51$ \cite{HKQ}.

\subsection{Local $\mathbb{P}^{2}$}

The holomorphic anomaly equations also apply to non-compact Calabi--Yau threefolds.
Let us consider the Calabi--Yau threefold $X^\vee=K_{\mathbb{P}^{2}}$, the total space of the canonical bundle of $\mathbb{P}^{2}$.
By varying the K\"ahler structure of $X^\vee$, we get a family $\mathcal{X}^\vee \to \mathcal{N}$.
The mirror family is constructed by following the lines in \cite{Chi} using Batyrev toric duality \cite{Bat}, or the Hori--Vafa construction \cite{HV}.
For definiteness, we will display the equation for the mirror family $\mathcal{X}\to \mathcal{M}$ obtained by the Hori--Vafa method
\begin{equation*}
uv-H(y_{1},y_{2};z)=0,\quad  (u,v,y_{1},y_{2})\in \mathbb{C}^{2}\times (\mathbb{C}^{*})^{2}\,,
\end{equation*}
where $H(y_{1},y_{2};z)=y_{1}y_{2}(-z+y_{1}+y_{2})+1$ and $z$ is the parameter for the base $\mathcal{M}$. 
It is a conic bundle over $(\mathbb{C}^\times)^2$ which degenerates along the so-called mirror curve $\{H(y_{1},y_{2};z)=0\}$. 
The mirror family $\mathcal{X}\to \mathcal{M}$ comes with the following Picard-Fuchs equation:
\begin{equation*}
\left(\theta^{3}-27z\theta (\theta+{1\over 3})(\theta+{2\over 3})\right)\phi=0.
\end{equation*}

Near the LCSL, given by $z=0$, there are three solutions of the form
\begin{equation*}
\phi^{0}(z)=1,\quad \phi^{1}(z)=\log z+\cdots,\quad \phi^{2}(z)=(\log z)^{2}+\cdots
\end{equation*}
and the mirror map is provided by $t(z)=\phi^{1}(z)/\phi^{0}(z)$.
As in the quintic case, the Yuwaka coupling can be solved from the Picard-Fuchs equation:
\begin{equation*}
C_{zzz}={\kappa\over z^{3}(1-27z)},
\end{equation*}
where $\kappa=-{1\over 3}$ is the classical triple intersection number of $X^\vee$.
The normalized Yukawa coupling in the $t$ coordinate is then
\begin{equation*}
K_{ttt}=(\phi^{0}(z))^{-2}({\partial z\over \partial t})^{3}{\kappa\over z^{3}(1-27z)}=-{1\over 3}{ (\theta t)^{3}\over 1-27z}\,.
\end{equation*}
From (\ref{soltogenus1hae}), the genus one amplitude is of the form
\begin{equation*}
\mathcal{F}_{1}={1\over 2}\log (G^{z\bar{z}}e^{K(4-{\chi(X^\vee)\over 12})})+ {1\over 2}|\log z^{b}(1-27z)^{a}|^{2}.
\end{equation*}
The constant $a$ is solved from the gap condition at the conifold point $z=1/27$ and turns out to be $a=-1/6$.
The constant $b$ has to satisfy the boundary condition at the LCSL given by
\begin{equation*}
{1\over 2}+{1\over 2}b=-{1\over 24}\int_{X^\vee} c_{2}(X^\vee)\cup H.
\end{equation*}
In the current case, we know $\chi{^\vee}=\chi(\mathbb{P}^{2})=3$ and $\int_{X^\vee} c_{2}(X^\vee)\cup H=2$ and thus we get at genus one
\begin{equation*}
\mathcal{F}_{1}={1\over 2}\log (G^{z\bar{z}}e^{{15\over 4}K})+ {1\over 2}|\log z^{-{7\over 6}}(1-27z)^{-{1\over 6}}|^{2}\,.
\end{equation*}
In the current non-compact case, we have the holomorphic limit by using (\ref{hollimitsofconnections})
\begin{equation*}
G_{z\bar{z}}\sim {\partial t\over \partial z},\quad K_{z}\sim -\partial_{z}\log \phi^{0}(z)=0\,.
\end{equation*}
Therefore we obtain
\begin{equation*}
\partial_{t}{\tt F}_1(t)=-{1\over 2}\partial_{t}\log {\partial t\over \partial z}
+ {1\over 2}\partial_{t}\log z^{-{7\over 6}}(1-5^{5}z)^{-{1\over 6}}\,.
\end{equation*}
The higher genus topological string amplitudes are more involved but can be worked out in a similar manner \cite{KZ2}.

%%%%%%%%%%%%%%%%%%%%%%%%%%%%%%%%%%%%%%%%%%%%%%%%%%%%%%%%%%%%%%%%%%%%%%%%%%%%%%%%%%%%%%
%%%%%%%%%%%%%%%%%%%%%%%%%%%%%%%%%%%%%%%%%%%%%%%%%%%%%%%%%%%%%%%%%%%%%%%%%%%%%%%%%%%%%%

\end{document}